\documentclass[12pt]{article}

\pdfoutput=1 
\usepackage[T1]{fontenc}
\usepackage{amsmath}
\usepackage{amsfonts}

\usepackage[english]{babel}

\usepackage{geometry}                
\usepackage[parfill]{parskip}    
\usepackage{graphicx}
\usepackage{amssymb}
\usepackage{epstopdf}
\usepackage[all]{xy}

\usepackage{amsthm} 
\newtheorem{thm}{Théorème}[section] 
\theoremstyle{définition}

\theoremstyle{remark} 
 
\theoremstyle{plain}

\theoremstyle{plain}

\begin{document}

\title{Descartes' Circle Theorem, Princess Elizabeth, and Spinors}

\author{Daniel Parrochia}
\date{University of Lyon (France)}
\maketitle

\begin{figure}[h] 
\vspace{-1\baselineskip}
\hspace{6\baselineskip}
\includegraphics[width=4in]{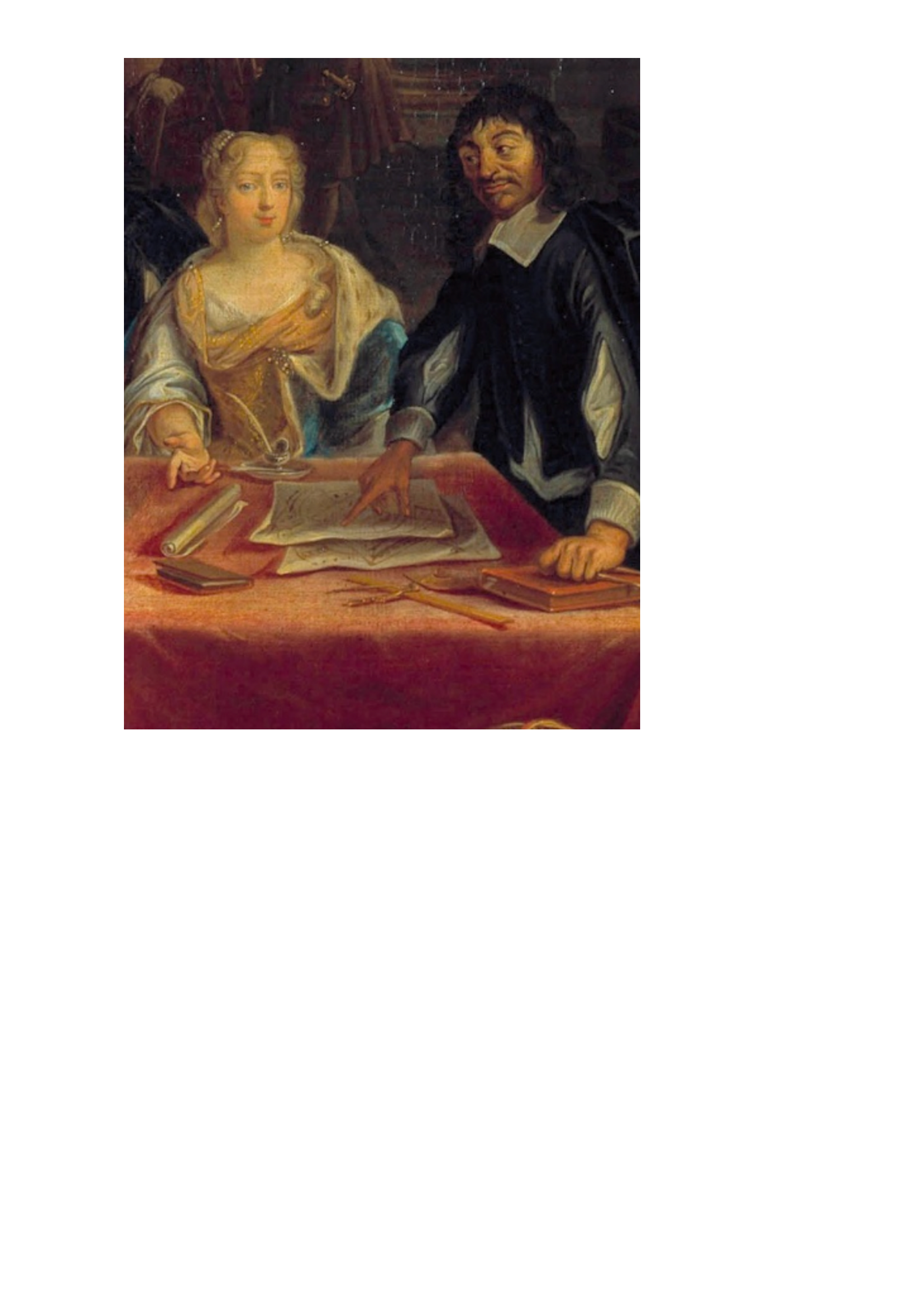}
\vspace{-12\baselineskip}
\caption{Descartes and Princess Elizabeth}
\label{fig: Elisa2}
\end{figure}
\vspace{1\baselineskip}

\textbf{}

\textbf{Abstract.}

In this article, on Descartes' famous "circle theorem," we first attempt to explain how the philosopher could have arrived at the equation that he presents without demonstration in a letter to Princess Palatine Elizabeth of Bohemia in November 1643 (a previously unpublished proof, to our knowledge). Then we report some stages of the subsequent generalization of this theorem, whose Cliffordian flavor, via the equivalence of a quadratic form and the square of a linear form, still mobilizes mathematicians today. This statement, which, over time, has undergone different extensions, Euclidean and non-Euclidean, has found, {\it in fine}, a spinor formalization. We see there the proof of what we may call, in a Bachelardian style, an "inductive value" of the truth, which extends by successive generalizations. We conclude, more briefly, with the contacts between Elizabeth and Descartes, and the perhaps symbolic meaning of these considerations on "kissing circles," as they are called in Anglo-Saxon countries, in the context of their correspondence on the soul and the body, and the question of passions.


\section{Brief History}

The problem we address dates back to the Greek geometer Apollonius of Perga (present-day Aksu in Turkey). Born in the second half of the 3rd century BC, Apollonius, who died in the early 2nd century BC, is known not only for his writings on conic sections, but also for a series of works to which Pappus of Alexandria provided guidance and whose contents Renaissance geometers have endeavored to rediscover. Among them is a book on "tangencies" ($'E\pi\alpha \phi \alpha \iota$ or De tactionibus), which includes the origin of the problem that will concern us and for which a remarkable generalization has recently been made.

This problem (the 10th in the series proposed by Apollonius) consists of finding a circle tangent to three others. This is the most difficult and interesting case considered by the mathematician in his work. François Viète, at the end of the 16th century, had already proposed this problem (called the "Apollonius problem") to Adrien Romain\footnote{Adrien Romain (1561-1615), whose real name was Adriaan van Roomen, Latinized as Adrianus Romanus then Frenchified as Adrien Romain, was a Belgian mathematician and doctor, an excellent calculator, notably remembered for having calculated the value of $\pi$ (he did not use this notation) with sixteen decimal places, 15 of which were exact. The solution that Romain proposed to the Apollonius problem that François Viète had submitted to him actually involved the construction of the intersection of two hyperbolas (see \cite{Bosm1}) and, as a result, did not respect the constraints of construction using a ruler and compass, unlike the one that Viète would quickly communicate to him by letter. This one is said to have written to Romain: "Eminent Adrian, as long as one touches the circle with hyperbolas, one does not touch it finely."}, in response to a challenge that the Belgian had launched (to solve an equation of degree 45) and that he, Viète, had successfully met. On the other hand, Adrian had only been able to solve Apollonius's problem by using auxiliary hyperbolas for the construction. On the contrary, Viète, faced with this same problem, had managed to find a solution "with ruler and compass" (that is, one that conformed to the requirements of the analysis of the Ancients), a solution he subsequently published in his book {\it Apollonius Gallus} (Paris, 1600).\footnote{The preface to the Camerer edition of the works of Apollonius (see \cite{Apo}) contains a detailed history of this problem in its early stages.}.

\section{Descartes' Circle Theorem}

Descartes, interested in the question, submitted the problem to the particularly astute Princess Elizabeth of Bohemia, daughter of Frederick V, the Elector Palatine, in exile in Holland and won over to his ideas. Two letters from November 1643 (see \cite{Bey}, 79-83) report this episode. In the first, the philosopher proposes his own solution, in the case where the non-intersecting circles are completely exterior to each other.

To precisely determine the central circle tangent to the other three, we need to know the respective distances between its center $D$ and the centers $A, B, C$ of the other three circles (see Fig. 2).

\begin{figure}[h] 
\vspace{-0.5\baselineskip}
\hspace{5\baselineskip}
\includegraphics[width=9in]{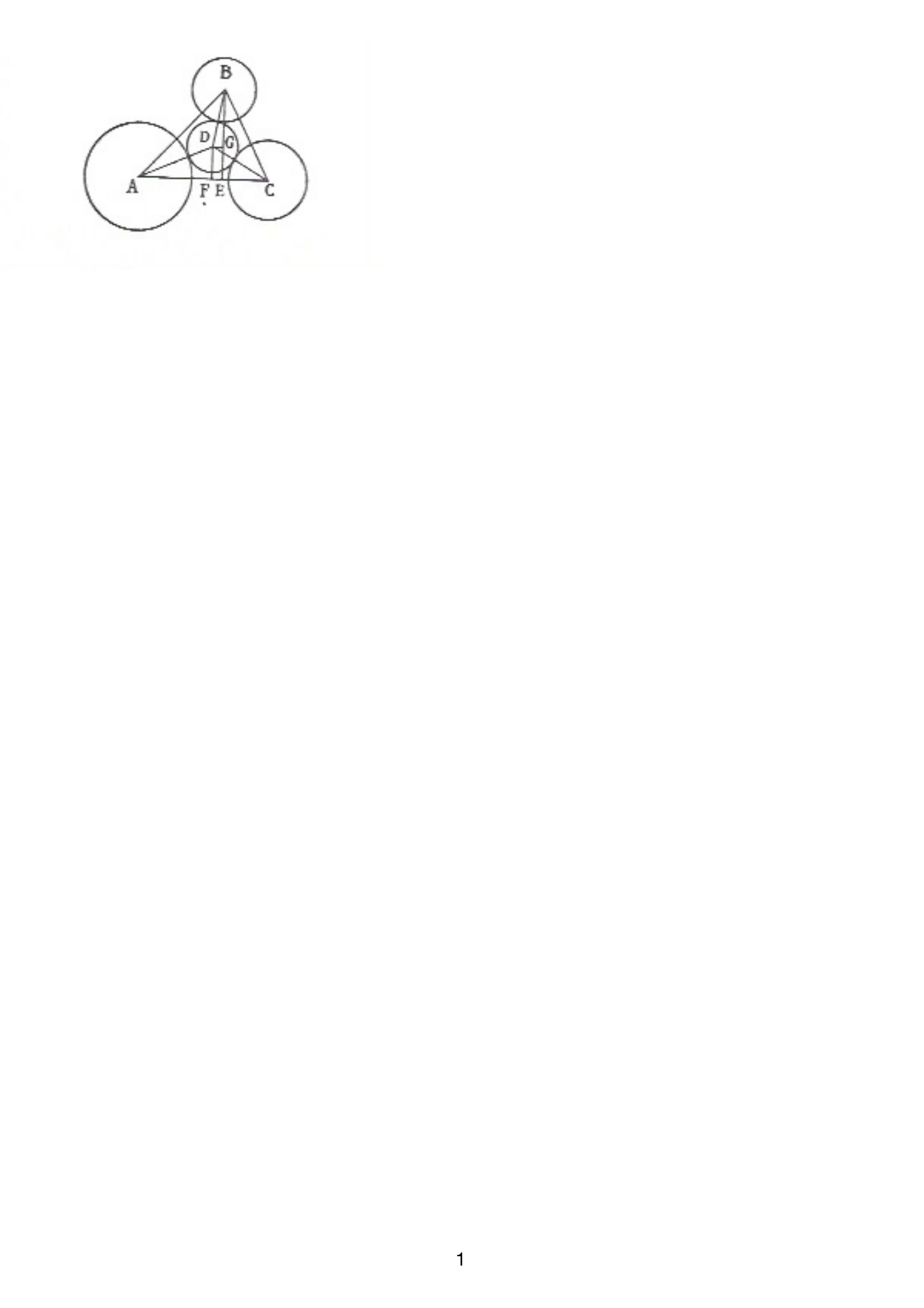}
\vspace{-53\baselineskip}
\caption{The Three Circles Problem}
\label{fig: cercle11}
\end{figure}

To do this, we consider the right triangles ADF, BDG, and CDF, and each time we apply the Pythagorean theorem relating the length of the hypotenuse to the sides of the right angle. We have:
\[
AD^2 = AF^2 + FD^2, \quad BD^2 = BG^2 + DG^2, \quad CD^2 = CF^2 + FD^2.
\]
Algebraizing the problem, Descartes introduces symbols, stating:
\[
AD = a + x, BD = b + x, CD = c + x, AE = d, BE = e, CE = f, DF = EG = y, DG = EF = z.
\]
Hence:
\[
AF = d-z, \quad BG = e - y, \quad CF = f + z.
\]
He obtains, for triangle ADF and the first Pythagorean formula (in modern notation):
\begin{equation}
a^2 + 2ax + x^2 = d^2 -2dz + z^2 + y^2,
\end{equation}
for triangle BDG and the second Pythagorean formula:
\begin{equation}
b^2 + 2x + x^2 = e^2 -2ey + y^2 + z^2,
\end{equation}
and for triangle CDF and the third Pythagorean formula:
\begin{equation}
c^2 + 2cx + x^2 = f^2 + 2fz + z^2 + y^2. \end{equation}
In order to eliminate unnecessary unknowns, he proposes subtracting (1) from (3) side by side, thus obtaining:
\[
c^2 + 2cx - a^2 - 2ax = 2fz - d^2 + 2dz,
\]
from which he can derive $z$.

Then, subtracting (2) from (1) (or from (3), which amounts to the same thing), and replacing $z$ with its value, he obtains:
\[
a^2 + 2ax - b^2 - 2bx = d2 - 2dz - e^2 + 2ey,
\]
from which he derives $y$.
Then returning to one of the first three equations and replacing $y$ and $z$ with their values, he could finally -- which he announces, but does not actually do -- derive $x$.

As Coxeter notes (see \cite{Cox}, 5), this method is clear but clumsy.

\section{Elizabeth's Solution and Cartesian Simplification}

\begin{figure}[h] 
\vspace{-1\baselineskip}
\hspace{-5\baselineskip}
\includegraphics[width=9in]{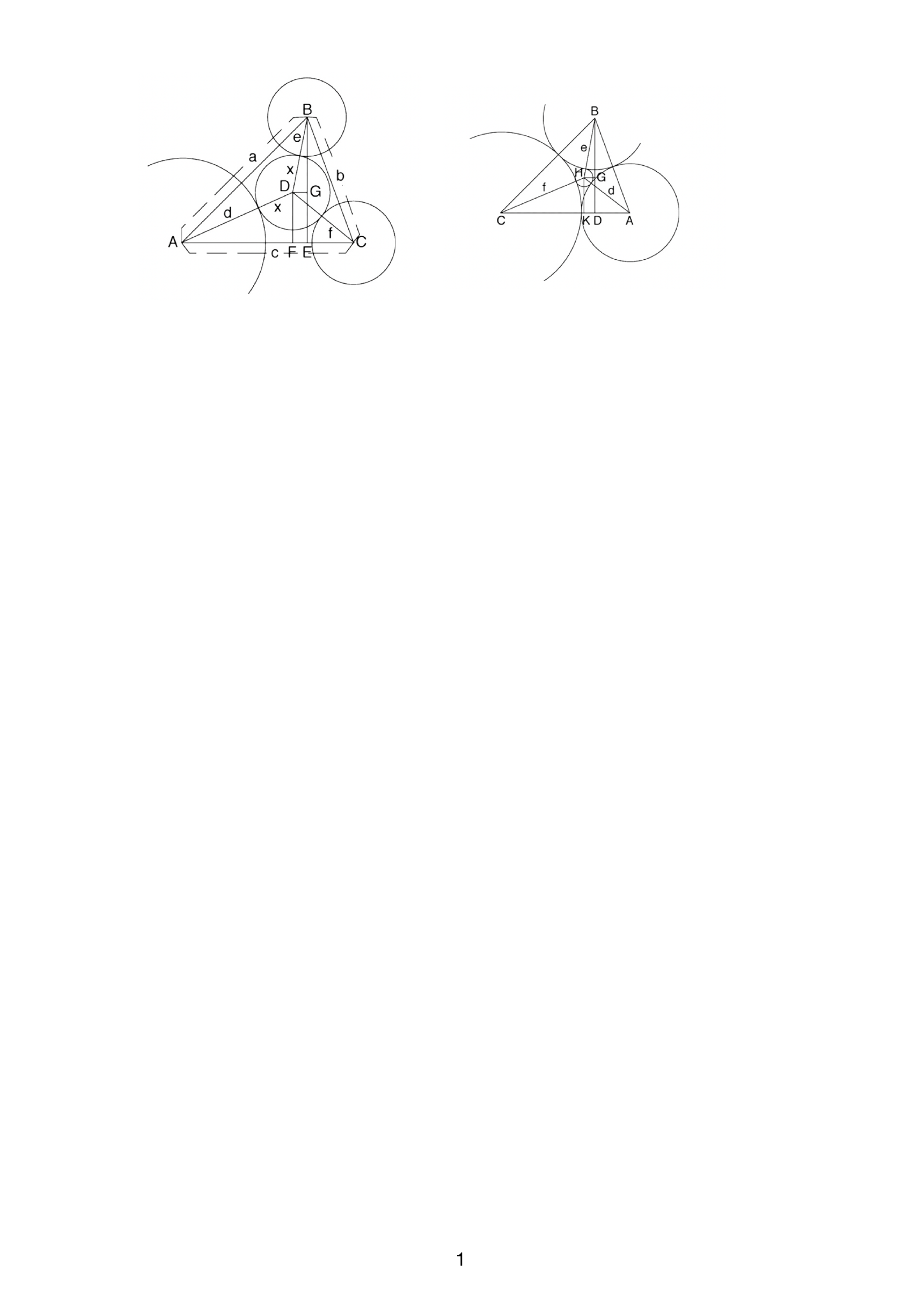}
\vspace{-49\baselineskip}
\caption{Elizabeth's Solution and Cartesian Simplification}
\label{fig: elisa11}
\end{figure}

Elizabeth, who had attempted to solve the problem on her own, had made a different choice of undeterminates and applied a different method. Fig. 3 (left) shows her diagram, and Fig. 3 (right) shows the simplification Descartes made by assuming that the outer circles touch.

In a second letter, which also draws on the method proposed by the Princess Palatine, he explains that this time he only needs to use the quantities $d, e, f$ to designate the radii of these circles and the unknown $x$ to designate the radius of the fourth circle, the one that touches them externally. Without saying so, Descartes is probably referring to Proposition 13 of Book II of Euclid's Elements: "In acute triangles, the square of a side opposite an acute angle is equal to the squares of the sides that include this acute angle minus twice the rectangle enclosed under the side of the acute angle on which the perpendicular drawn from the vertex of the opposite angle falls, and under the straight line intercepted between this acute angle."

Applying this proposition to the case of triangle AHC, we can state:
\[
HC^2 = HA^2 + AC^2 - 2CA.AK, \quad \textnormal{or:} \quad AK = \frac{HA^2 + AC^2 - HC^2}{2CA}. \]
Knowing that:
\[
HC = f+x, \quad HA = d+x, \quad AC = f+d,
\]
we have:
\[
AK = \frac{(d+x)^2 + (f+d)^2 - (f+x)^2}{2(d+f)}.
\]
A simple calculation then allows us to find the first result posed by Descartes in this second letter:
\[
AK = \frac{d^2 + df + dx - fx}{d+f}.
\]

Now applying Euclid's same proposition to triangle $ABC$, we obtain, in a similar manner:
\[
BC^2 = BA^2 + AC^2 - 2CA.AD, \quad \textnormal{i.e.:} \quad AD = \frac{BA^2 + AC^2-BC^2}{2CA}. \]
and knowing that:
\[
BC = f+e, \quad BA = d+e, \quad AC = f+d,
\]
\[
AD = \frac{d^2 + df + de - fe}{d+f}.
\]

The question that now arises is how we move from these initial results to the following equation stated by Descartes without proof, and as if dropped from the sky:
\begin{align}
d^2e^2f^2 & + d^2e^2x^2 +d^2f^2x^2+e^2f^2x^2 \notag \\
& = 2def^2x^2+2de^2f^2x+ 2de^2fx^2 + 2d^2ef^2x+2d^2efx^2+2d^2e^2fx.
\end{align}

Apparently, no one has solved this mystery yet. Anne Koyé, a French mathematician specializing in the problem of Apollonius, who wrote a thesis on the subject under the supervision of Jean Dhombres, says nothing on the matter, merely stating that "it is not so obvious to find the transition from the first equalities to the last" (see \cite{Boy}, 64)\footnote{Anne Boyé also believes that these "first equalities" were reached by means of the Pythagorean theorem, whereas they most certainly result, as we have said, from proposition 13 of Book II of Euclid's {\it Elements}.}. More than twenty years ago, using a modern algebra computer program, the Dutch historian of mathematics Henk Bos had shown that it was possible, in the case of Apollonius's initial problem where the three peripheral circles do not touch, to reduce the three equations with three variables $x, y$ and $z$ (those that Descartes presents in his first letter), to a single quadratic equation with a single variable (as recommended in his {\it Géométrie}), but with 87 terms! In practice, solving such an equation would obviously have been difficult for \'{E}lisabeth, or even for Descartes himself. Even with the latter's method, and the reduction of the data to the three symbols $d, e, f$ and the single unknown $x$, the solution, this time involving 78 terms, would remain difficult to access. For the simplified case where the peripheral circles touch, the problem is somewhat simpler. Henk Bos, however, has proposed a solution that does not follow Descartes' suggestions in his second letter, but in fact amounts to using his first method (see \cite{Bos1}, 205-209)\footnote{Bos explains that the most direct way he has found to prove the theorem is to express the Cartesian indeterminates $d, e, f$ in terms of the radii $a, b, c$. Since these terms $a, b, c$ no longer exist in the second letter, it is clear that Bos is returning to the first, as he explicitly states, since it involves reinserting the formulas:
\[
d = \frac{a^2+ac+ab-bc}{a+c}, \quad e = 2\sqrt{\frac{abc)(a+b+c}{a+c}}, \quad f = \frac{c^2+bc+ac - ab}{a+c}
\]
in the results stated by Descartes and numbered here (1)-(3), as well as those deduced from them. But, as Bos acknowledges, the calculations in this case, although within Descartes' grasp, are considerable. Bos also acknowledges that, in his second letter, Descartes actually gives the impression of wanting to follow Elizabeth's approach rather than his own. "In this case," he asserts, "the derivation of the equation is a little less simple, although not insurmountable" (see \cite{Bos1}, 209, note). But he does not give it.}. Now there is a real question here: why did the philosopher state, in his second letter, these first results (on $AK$ and $AD$), which point in the direction of calculating the areas of triangles, if not to pursue them? The enigma remains, including for the American mathematician and science journalist Dana Mackenzie who, in a recent article, limits himself to saying that Descartes, guided towards the right solution by Elizabeth, "in one way or another" realized in eight days or less how the problem was simplified when the circles, as they said at the time, "touch" (see \cite{Mac}, 76). But we don't know any more about the hidden proof...

One appropriate way to approach the question is to take seriously what Descartes wrote to Elisabeth, as cited earlier. Here is the proof I have obtained:

\begin{proof}

Given the values of $AK$ and $AD$, we can indeed calculate $KD = AD - AK$. Let us then consider the right-angled triangle $BHG$ in Figure 3 (right). It turns out that the lengths of all its sides are known (or can be calculated). We have:
\begin{enumerate}
\item $HG = KD = AK - AD.$
\item $BH\ (\textnormal{hypotenuse}) = e + x.$
\item $BG = BD - GD = BD - HK$.
\item $HK^2 = HC^2 - KC^2 = (f+x)^2 - KC^2 = (f+x)^2 - (AC-AK)^2$.
\end{enumerate}
We can therefore apply the Pythagorean theorem to triangle $BHG$, yielding:
\[
KD^2 + BG^2 =(e+x)^2.
\]
Or, substituting the segments with their values:
\begin{equation}
(AK-AD)^2 +  (BD - HK)^2 = (e+x)^2.
\end{equation}

Expanding the expression $(BD - HK)^2$ and substituting certain segments with their values in (5), we obtain:
\[
KD^2 + (d+e)^2 - AD^2 + (d+x)^2 - AK^2 - 2BD.HK =(e+x)^2.
\]
Since the product $2BD \cdot HK$ has the disadvantage of introducing square roots, we eliminate them by isolating this double product on one side of the equation, thereby obtaining:
\[
2BD \cdot HK = KD^2 + (d+e)^2 - AD^2 + (d+x)^2 - AK^2 - (e+x)^2.
\]
We shall see that by expanding the squares of the radii $(d+e)^2, (d+x)^2, (e+x)^2$ and recalling that $KD = AK - AD$, the right-hand side of this equation simplifies significantly.

Let us call this right-hand side $\Omega$. We have:
\[
\Omega =KD^{2}+(d+e)^{2}-AD^{2}+(d+x)^{2}-AK^{2}-(e+x)^{2}.
\]
We know that $KD = AD-AK$. Let us therefore substitute $KD$ with its value and expand its square. Then, let us substitute this into $\Omega$. After cancelling out the terms that sum to zero, we obtain:
\[
\Omega = -2AK \cdot AD + (d+e)^2 + (d+x)^2 - (e+x)^2.
\]
By expanding the algebraic identities for the radii and simplifying their sum, we arrive at the expression:
\[
\Omega = 2 d^2 + 2de + 2dx - 2ex -2AK \cdot AD = 2(d^2 + de + dx - ex - AK \cdot AD).
\]

Let us substitute the fractional expressions for \(AK\) and \(AD\) given by Descartes into this simplified term \(\Omega \).

As a reminder, our expressions are: \(AK = \frac{d^2 + df + dx - fx}{d+f}\) and \(AD = \frac{d^2 + df + de - fe}{d+f}\).

The simplified term from the preceding reasoning can now be written as:
\[
\Omega =2\cdot \left[d^{2}+de+dx-ex-AK\cdot AD\right].
\]

Let us now replace (step 1) the product \(AK \cdot AD\) with its fractional form. The denominator of this product will be \((d+f)^2\). To group everything under this same fraction, we must multiply the isolated terms by \((d+f)^2\). We obtain:
\[
\Omega =2\cdot \left[\frac{(d^{2}+de+dx-ex)(d+f)^{2}-(d^{2}+df+dx-fx)(d^{2}+df+de-fe)}{(d+f)^{2}}\right].
\]
Let us now focus (step 2) solely on the giant numerator inside the brackets (let us call it \(N_{\Omega }\)). 

First, let us expand \((d+f)^2 = d^2 + 2df + f^2\), then multiply it by the first block \((d^2 + de + dx - ex)\). By distributing term by term, we obtain 12 monomials:
\begin{align*}
(d^{2}+de+dx-ex)(d^{2}+2df+f^{2})= d^{4}+2d^{3}f+d^{2}f^{2}+d^{3}e & \\
+ 2d^{2}fe+df^{2}e+d^{3}x+2d^{2}fx +df^{2}x-d^{2}ex-2dfex-f^{2}ex.
\end{align*}

Now let us expand (step 3) the product of Descartes' two original numerators: \((d^2 + df + dx - fx)(d^2 + df + de - fe)\). To simplify this multiplication of two 4-term polynomials (which would normally yield 16 terms), let us note that they share a common block: \((d^2 + df)\). Let us temporarily set \(P = d^2 + df\). The product becomes:
\[
[P+(d-f)x]\cdot [P+(d-f)e]=P^{2}+P(d-f)e+P(d-f)x+(d-f)^{2}ex.
\]
Substituting \(P = d(d+f)\) and \((d-f)\) back in, we expand to obtain these 12 monomials (applying the overall minus sign in front of this second term), yielding:
\begin{align*}
-(d^{2}+df+dx-fx)(d^{2}+df+de-fe)= -d^{4}-2d^{3}f-d^{2}f^{2}-d^{3}e & \\-d^{2}fe+d^{2}fe+df^{2}e
-d^{3}x-d^{2}fx+d^{2}fx+df^{2}x-d^{2}ex+2dfex-f^{2}ex.
\end{align*}

When adding (step 4) the term from step 2 and the term from step 3 to form the complete numerator \(N_{\Omega }\), an impressive number of high-degree terms cancel each other out. Let us look at the immediate simplifications:

- The fourth powers cancel out: \(d^4 - d^4 = 0\)

- The \(f\)-cubes cancel out: \(2d^3f - 2d^3f = 0\)

- The \(d^{2}f^{2}\) terms cancel out: \(d^2f^2 - d^2f^2 = 0\)

- The \(e\)-cubes cancel out: \(d^3e - d^3e = 0\). - The \(x^3\) terms cancel out: \(d^3x - d^3x = 0\)

By grouping and adding the remaining terms (notably the cross-terms like \(2dfex\) and \(-2dfex\), which offset each other given the signs), the numerator \(N_{\Omega }\) simplifies dramatically to a single expression:
\[
N_{\Omega }=2d^{2}fe+2d^{2}fx+2df^{2}e+2df^{2}x-2d^{2}ex-2f^{2}ex.
\]
By factoring out \(2df\), or by rearranging, we see that this block condenses perfectly into:
\[
N_{\Omega }=2df(d+f)(e+x)-2d^{2}ex-2f^{2}ex,
\]
the final result for \(\Omega \) prior to squaring.

Reintroducing the external factor of 2 and the denominator \((d+f)^2\), our complete \(\Omega \) expression becomes:
\[
\Omega =\frac{4df(d+f)(e+x)-4(d^{2}+f^2)ex}{(d+f)^2}.
\]

We now have an expression for \(\Omega \) in the form of a proper fraction, free of any extraneous fourth or third powers in the numerator. When we equate the square of this expression (\(\Omega ^{2}\)) with the left-hand side (\(4 \cdot BD^2 \cdot HK^2\)) multiplied by \((d+f)^4\), all denominators will vanish. The remaining terms on the left and right will combine to form the famous final symmetric equation of degree 4, containing terms such as \(d^{2}e^{2}f^{2}\), \(d^{2}e^{2}x^{2}\), etc., along with the double-product terms featuring the coefficient 2.

To explicitly set up the final equation between the two sides, we square the simplified expression for \(\Omega \) and equate it to the left-hand side (\(4 \cdot BD^2 \cdot HK^2\)).

To avoid carrying around cumbersome fractions, we first multiply both sides by \((d+f)^4\) to eliminate all denominators once and for all.

Here is the final calculation leading directly to Descartes' formula.

Recall that the original left-hand side was \(4 \cdot BD^2 \cdot HK^2\). By replacing \(BD^{2}\) and \(HK^{2}\) with their expanded values -- expressed using the fractions involving \(AD\) and \(AK\) -- and then multiplying by \((d+f)^4\), we obtain, after expanding the internal blocks:
\[
\mathbf{LHS}_{\mathbf{final}}= 4d^{2}f^{2}\left[(d+f)^{2}(d+e)(f+e)(d+x)(f+x)-\dots \right].
\]
Symbolic calculation shows that this left-hand term, once cleared of fractions, condenses into the following raw polynomial form:
\[
\mathbf{LHS}_{\mathbf{final}}=4d^{2}f^{2}\left[d^{2}e^{2}+d^{2}x^{2}+f^{2}e^{2}+f^{2}x^{2}+2def(d+e+f+x)+\dots \right].
\]

The right-hand side corresponds to the square of our numerator \(N_{\Omega }\), obtained in the previous step. Recall the simplified form of this numerator:
\[
N_{\Omega }=2d^{2}fe+2d^{2}fx+2df^{2}e+2df^{2}x-2d^{2}ex-2f^{2}ex.
\] 
We can factor out \(2\) to simplify the squaring process:
\[
N_{\Omega }=2\cdot \left[d^{2}f(e+x)+df^{2}(e+x)-(d^{2}+f^{2})ex\right].
\]
Upon squaring (\(RHS_{final} = N_{\Omega}^2\)), the coefficient \(2^{2}\) becomes 4. We obtain:
\[
\mathbf{RHS}_{\mathbf{final}}= 4\cdot \left[d^{2}f(e+x)+df^{2}(e+x)-(d^{2}+f^{2})ex\right]^{2}.
\]
We set up the equality:
\[
\mathbf{LHS}_{\mathbf{final}}\mathbf{=RHS}_{\mathbf{final}}.
\]

Since the factor 4 appears on both sides, it cancels out immediately through division:
\[
d^{2}f^{2}\left[\dots \text{Block\ G}\dots \right]=\left[d^{2}f(e+x)+df^{2}(e+x)-(d^{2}+f^{2})ex\right]^{2}. 
\]
It is at this precise stage that the algebraic "magic" observed by Descartes occurs. When both sides are fully expanded, terms of total degree 6 (such as \(d^{4}f^{2}e^{2}\)) appear on either side. By grouping all terms on one side of the equation (\(LHS - RHS = 0\)), it becomes possible to factor out \(d^{2}f^{2}\) (or perform an equivalent rearrangement) across the entire remaining polynomial. After this final division by the trivial common factors, the equation reduces strictly to the combined symmetric terms of degree 4:

\(d^{2}e^{2}f^{2}+d^{2}e^{2}x^{2}+d^{2}f^{2}x^{2}+e^{2}f^{2}x^{2}-2def^{2}x^{2}-2de^{2}f^{2}x-2de^{2}fx^{2}-2d^{2}ef^{2}x-2d^{2}efx^{2}-2d^{2}e^{2}fx=0\),

in other words, Descartes's final formula.
\end{proof}

We have thus succeeded in explaining equation (4) -- which had been "dropped in" without explanation in the letter to Elisabeth -- by following only Descartes's suggestions based on the expressions for $AK$ and $AD$, and by applying the tedious method known only to Descartes\footnote{In earlier versions of this text (see \cite{Par3}), we proposed a more elegant proof (based on Heron's formula), but it presupposed (more or less) the result would have been already known. It goes without saying that the modern proof of Descartes' theorem does not concern itself with the calculations mentioned in this article or even here. In fact, to avoid the algebraic nightmare of successive squaring of Heron's formula, modern geometry uses an extraordinary tool: the Cayley-Menger determinant. Instead of expanding giant polynomials, the distance constraint of the four tangent centers is written in a 5×5 matrix. Since the volume of a simple triangle in a plane is zero if an inconsistent superfluous dimension is added, the determinant of this matrix must be equal to 0. The matrix takes this simplified form:
\[
\left|\begin{matrix}0&(d+e)^{2}&(d+f)^{2}&(d+x)^{2}&1\\ (e+d)^{2}&0&(e+f)^{2}&(e+x)^{2}&1\\ (f+d)^{2}&(f+e)^{2}&0&(f+x)^{2}&1\\ (x+d)^{2}&(x+e)^{2}&(x+f)^{2}&0&1\\ 1&1&1&1&0
\end{matrix}\right|=0.
\]
Calculating a 5×5 determinant remains tedious, but the symmetry properties of the rows and columns (factorizations by sums of radii) allow the result to be condensed almost magically to arrive at Descartes' equation without having to manipulate isolated monomials.}.

That said, did the philosopher actually perform this calculation himself, or did he merely recall the result presented in his first letter of November 1643 -- where he placed the circles in a Cartesian coordinate system, wrote down the distance relationships between their centers, and manipulated the equations to eliminate the spatial variables $x$ and $y$ in order to arrive at a relationship between the radii? Upon discovering Elisabeth's method -- which had stalled in the (overly complex) general case but began with a good choice of unknowns -- he may simply have indicated to her how to reach the goal in the simplified case, without actually redoing the calculations himself.

Coxeter (see \cite{Cox}, 5), who shows that equation (4) can now be expressed in the more concise form:
\begin{equation}
\frac{1}{d^2}+ \frac{1}{e^2}+ \frac{1}{f^2}+ \frac{1}{x^2} = \frac{2}{ef}+ \frac{2}{fd}+ \frac{2}{de}+ \frac{2}{dx}+ \frac{2}{ex}+ \frac{2}{fx} .
\end{equation}
or in the even more beautiful form of the following one:
\[
(\kappa_1+\kappa_2+\kappa_3+\kappa_4)^2 = 2(\kappa_1^2+\kappa_2^2+\kappa_3^2+\kappa_4^2),
\]
finds it strange that Descartes did not do so. But precisely, he probably did not see this possibility.

In fact, it will be noted that introducing variables of the type $\kappa_i$ to express the inverse of the radii practically presupposes the concept of "curvature", associated with a relation of the type $1/r$ for circles, and which will only appear with Newton\footnote{Although Nicole Oresme used the word "curvitas" in Latin and declared that the curvature of a circle is "uniformus" and that this curvature is proportional to the inverse of its radius, it will be necessary to wait for Newton, in problem 5 of his {\it Methods of series and fluxions} to read that a circle has a constant curvature, inversely proportional to its radius, and the end of the century, in England, to see the term applied to surfaces.}. Descartes never really achieved this magnificent result, which Pedoe (see \cite{Ped}, 634) generously calls "Descartes' circle theorem", since his objective was not -- and for good reason -- to find a relationship expressing the links between the "curvature" of the different circles, but quite simply, as we said above, to calculate the quantity $x$ allowing him to then fix exact distances between their centers.

 It remains that these solutions of different spirit allow just as much the effective construction of what the English call, figuratively, "kissing circles". That said, the most interesting aspect, which is bound to escape philosophers with little mathematical understanding, is the Cliffordian flavor of equation (4), which is simply the equating of a quadratic form with the square of a linear form, one of the characteristics of Clifford algebras that will allow for unprecedented developments.

Having settled the historical point, let us now show how "Descartes's theorem" was progressively generalized, to the point of recently involving highly sophisticated mathematics.

\section{From Descartes' "kissing circles" to multidimensional flowers}

Without attempting to recount in detail the history or successive developments of Descartes's theorem over the following centuries, let us nevertheless provide a few landmarks here that will help the reader get their bearings.

The correspondence between Descartes and Princess Elizabeth of the Palatinate has been the subject of several editions or commentaries in English-speaking countries (see Bos, Nye, 31-33], Sha, 37-38, 73-81) and has also been republished in French by Michelle and Jean-Marie Beyssade (see \cite{Bey}), without, however, addressing the content of this theorem or the reasons that might have prompted Descartes to present the exercise to the princess. From a strictly mathematical perspective, according to Daniel Mathews and Orion Zymaris (see \cite{Mat}, 3), equation (4) was reformulated by the Japanese Yamaji Nushizumi in 1751 (see \cite{Mic1}, \cite{Mic2}) and then rediscovered several times, notably by Steiner in 1826 (see \cite{Ste}), Beecroft in 1842 (see \cite{Bee}), and soon after by Soddy (in 1936) (see \cite{Sod2})\footnote{This is Frederick Soddy (1877-1956), a British radiochemist, not a mathematician, who won the Nobel Prize in Chemistry in 1921.}. The latter, oddly enough, stated the result in a poetic form.

Numerous generalizations of Descartes' theorem are also known. For example, in 1936, Thorold Gosset extended this theorem, which initially holds for circles, to $n + 2$ mutually tangent spheres in $n$ dimensions, thus adding a line to Soddy's poem (see \cite{Gos}).

{\it "And let us not confine our cares\\
To simple circles, planes and spheres,\\
But rise to hyper flats and bends\\
Where kissing multiple appears,\\
In n-ic space the kissing pairs\\
Are hyperspheres, and Truth declares,\\
As n + 2 such osculate\\
Each with an n + 1 fold mate\\
The square of the sum of all the bends\\
in n  times the sum of their squares"}.

In 1962, J. G. Mauldon further generalized the above thinking to spherical and hyperbolic space (see \cite{Mau}). Then in 2002, Lagarias and his colleagues (see \cite{Lag}) extended these results by relating not only the curvatures, but also the centers of the spheres involved. Finally, in 2007, Kocik (see \cite{Koc1}) extended the theorem to $n + 2$ Euclidean spheres in $n$ dimensions in the case where the spheres in question need not be tangent.

If we now call the configuration of the three Apollonian circles a "3-flower"\footnote{Informally (more details will be provided later), we will say that an $n$-flower is an ordered configuration consisting of a central circle and $n$ petal circles all around, such that these are externally tangent to the central circle, and each is externally tangent to the one before and after it.}, it is clear that, in general, stacks of Apollonian circles (see Fig. 4) consisting of nested 3-flowers are, even today, a field of mathematical research in their own right (see in particular \cite{Aha} or \cite{Step} for the general context). These stacks have important properties in number theory and have been the subject of recent advances (see, for example \cite{Haa}). A special case related to work that we will soon comment on was reported by Daniel Mathews in (\cite{Mat1}). This is the case of Ford circles (see \cite{For}) which arise from integer spinors (see \cite{For})\footnote{"Ford circles" are named in honor of the American mathematician Lester Ford (senior), who described them in a paper published in 1938. A circle is placed on each irreducible fraction, for example 0/1, 1/1, 1/2, 1/3, 2/3, 1/4, 3/4, 1/5, 2/5, 3/5, 4/5. Each circle will be tangent to the x-axis as well as to neighboring circles. Fractions with the same denominator have circles of the same size.}.

 \begin{figure}[h] 
\vspace{-1\baselineskip}
\hspace{7\baselineskip}
\includegraphics[width=7in]{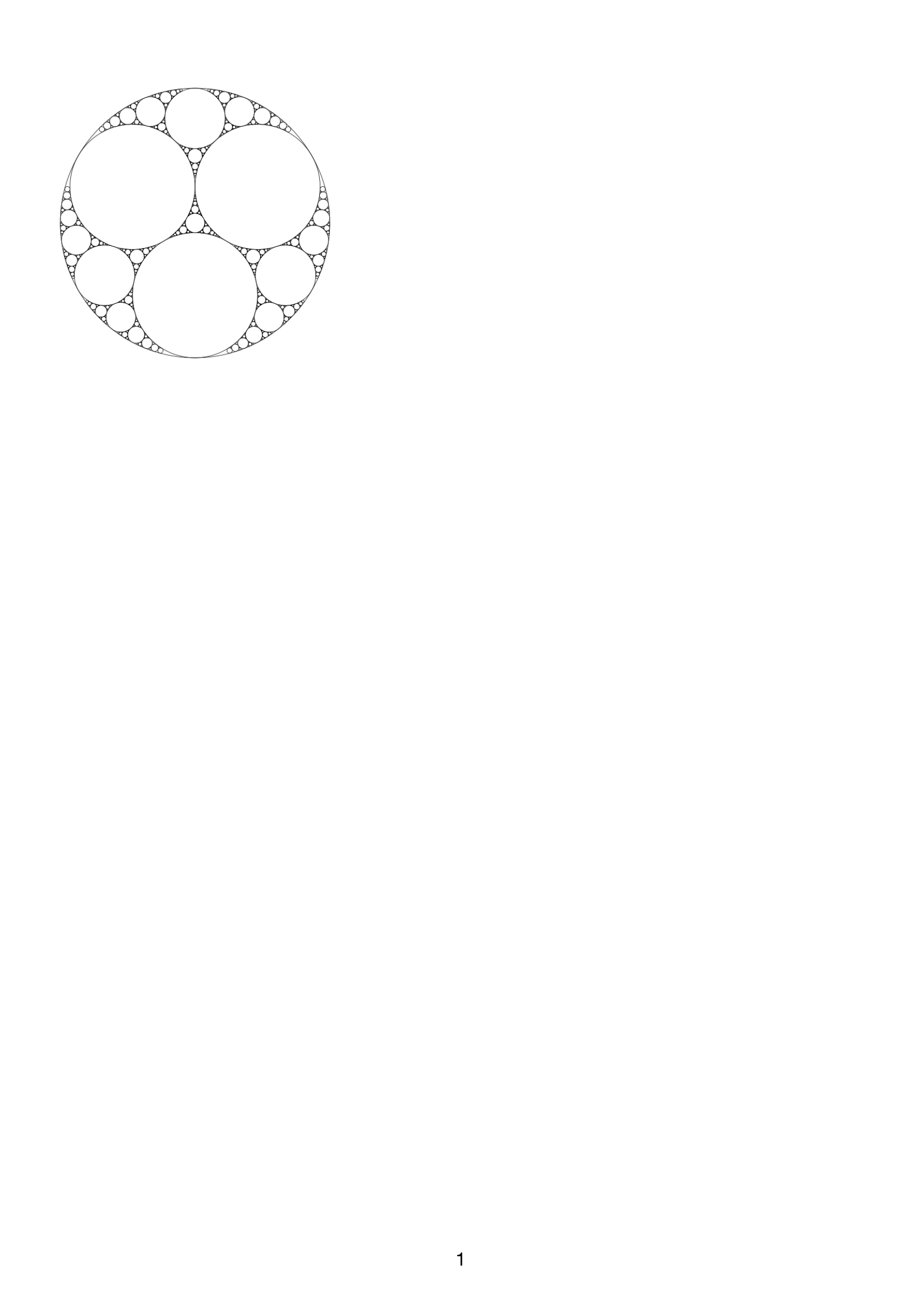}
\vspace{-37\baselineskip}
\caption{Apollonian Gasket}
\label{fig: Baderne11}
\end{figure}

The important point for us, however, is that, in addition to this approach, many other works such as those of Kocik (see \cite{Koc1}, \cite{Koc2}, \cite{Koc3}, \cite{Koc4}, \cite{Koc5}) also use spinors to describe Cartesian circle configurations and Apollonian stackings. In these works, spinors are complex numbers (defined up to the sign) describing the tangents between pairs of circles. We will show later how Descartes's theorem is deduced from these spinor considerations.

But we can already draw a provisional conclusion. Against the growing skepticism regarding knowledge, which is indeed changeable because it is obviously situated in history, Bachelard once insisted that certain truths are "forever" in science. And not content with remaining what they are, that is, fixed once and for all, they give rise to generalizations and extensions. For example, Archimedes' principle, which initially applies to fluids, was later applied to gases. Or again, Einstein's theory of relativity, first stated in a restricted form and applied to Galilean motions (rectilinear and uniform), was later generalized to accelerated motions, and then further extended by Weyl and other physicists (see \cite{Par1}). In other words, there is indeed an "inductive value" of true statements that gradually creates a sanctioned and therefore cumulative history, beyond all paradigm shifts and all imaginable "ruptures."

\section{Descartes and the Spinors}

Let us return, however, to Descartes and the circle theorem. Despite new proofs (including Euler's (see \cite{Eul}) (trigonometric) of formula (4), rediscovered, as we have seen, by Steiner, Beecroft, Coxeter, and Pedoe (who listed other proofs), it was not until the early 2000s that the depth of this theorem and its extension possibilities were truly appreciated.

\subsection{New cases, new writings}
From this perspective, the work of Lagarias and his colleagues (see \cite{Lag} then \cite{Gra}) was decisive, paving the way for the later thinking of Kocik (see \cite{Koc1}) and then Mathews (see \cite{Mat} and \cite{Mat1}) on spinors. Let us briefly review how this generalization was achieved.

First, the introduction of the concept of "curvature" invites us to see that Descartes' work in fact only concerned one of the possible configurations likely to be associated with the circle theorem (see Fig.5).

 \begin{figure}[h] 
\vspace{-1\baselineskip}
\hspace{4\baselineskip}
\includegraphics[width=9in]{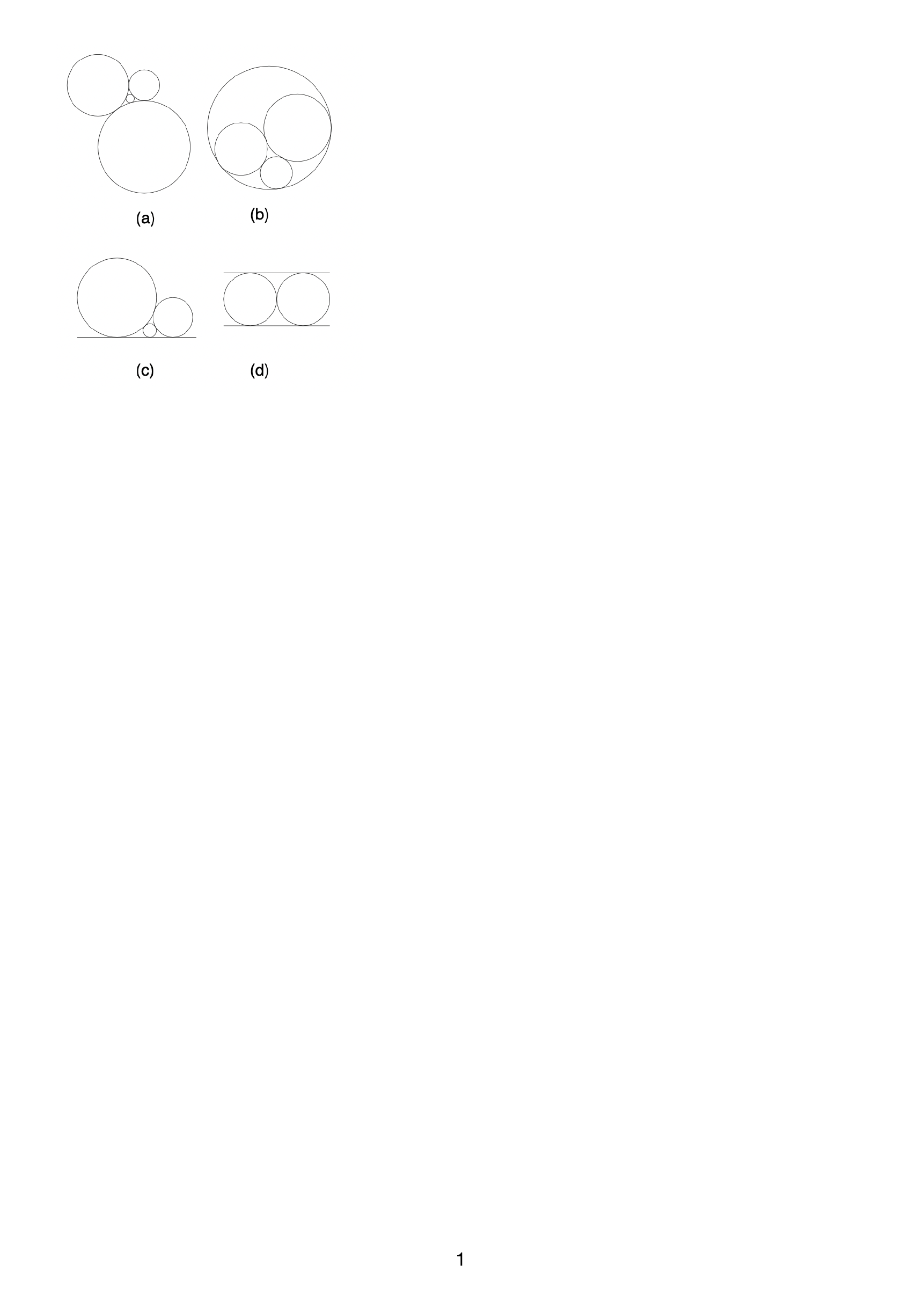}
\vspace{-46\baselineskip}
\caption{The different configurations of tangent circles according to \cite{Mat1}}
\label{fig: circle11}
\end{figure}

Descartes only had case (a) in mind, while another circular solution is (b), cases (c) and (d) relating to degenerate tangent circles reduced to straight lines\footnote{Lagarias (see \cite{Lag}) observes that given three mutually tangent circles, with curvatures $b_1, b_2$ and $b_3$, there are exactly two other circles tangent to each of them, each giving a four-circle Descartes configuration. The curvatures of these two new circles are the roots of the quadratic equation
\[
b_4 + b'_4 = 2(b_1 + b_2 + b_3). \quad (\ast)
\]
Thus, from a Descartes configuration, one can select any of the four circles and replace it with the other circle tangent to the other three; this gives a new Descartes configuration. The new curvature can be obtained from the initial four using ($\ast$). This construction can be repeated indefinitely.}.

But the first step in the true generalization is to rewrite the "Cartesian" formula (4) in the form:
\[
\sum_{j=1}^4 b_j^2= \frac{1}{2}(\sum_{j=1}^4 b_j)^2.
\]
then to extend it, via oriented configurations, to spheres of $n$ dimensions, so that we have, at that point, the relation:
\[
\sum_{j=1}^{n+2} b_j^2= \frac{1}{n}(\sum_{j=1}^{n+2} b_j)^2. \]
Reiterating the Cartesian construction in each of the circles surrounding the central circle leads to a fractal-like Apollonian stack, sometimes called "Apollonian gasket."\footnote{Translated into French as "baderne" by François Apéry (see \cite{Ape}). Starting from the Cartesian configuration, by continuing this construction step by step in this manner, we can add $2 \times 3^n$ new circles at step $n$, giving a total of $3^{n+1}+2$ circles after $n$ steps. In 1943, it was possible to prove that the residual set, complementary to the disks in the large disk, has a zero area (see \cite{Kas}), and therefore that the stacking is complete. In 1973, it was shown to have a Hausdorff dimension equal to approximately 1.3057.}, moreover known to Leibniz (see \cite{Lei1}, \cite{Lei2})\footnote{It is rather curious to note that Michel Serres (see \cite{Ser}, I, 371), in his famous thesis on the Leibniz system and its mathematical models, does not faithfully reproduce the figure drawn by Leibniz, precisely forgetting the central circle tangent to the other three, and which clearly showed that the diagram was closer to Apollonius and Descartes than to Leeuwenhoeck's microscope and the question of the nesting of the germs with which he seeks to bring it closer. On the other hand, B. Mandelbrot (see \cite{Man}, 243) cites Leibniz's letter to the R.F. des Bosses dated March 11, 1716, as one of the first manifestations of fractal thinking.}.

But the generalization continues with the introduction of complexes\footnote{In this case, the relation ($\ast$) previously mentioned in the footnote transforms into the relation:
\[
b_4z_4 + b'_4z_4' = 2(b_1z_1 + b_2z_2 + b_3z_3).
\]
and the general formula including the complex equations takes the form indicated in the text.}, which allows us to posit, for four mutually tangent circles, of curvature $b_j$ and center $z_j = x_j + iy_j$
\[
\sum_{j=1}^{4} (b_jz_j)^2= \frac{1}{2}(\sum_{j=1}^{4} b_jz_j)^2. \]
Lagarias and his colleagues (see \cite{Lag}, 6) were the first to express the Cartesian formula in a matrix form, which Kocik would later reuse. If we write Descartes' quadratic form in the generic form:
\[
b_1^2+b_2^2+b_3^2+b_4^2 = \frac{1}{2}(b_1+b_2+b_3+b_4)^2,
\]
we can make it correspond to the following matrix notation:
\[
\mathbf{Q}_2 : = I_4 - \frac{1}{2} \mathbf{I_4I_4^T}= \frac{1}{2}
\begin{bmatrix}
-1 &  1 & 1 & 1 \\
1 & -1 & 1 & 1 \\
1 & 1 & -1 & 1 \\
1 & 1 & 1 & -1
\end{bmatrix}
\]
where $\mathbf{I}_n$ denotes a column of $n$ 1 and $\mathbf{Q}_2$ is the coefficient matrix of the Cartesian quadratic form\footnote{The matrix 
$\mathbf{Q}_2 : = I_4 - \frac{1}{2} \mathbf{I_4I_4^T}$ has eigenvalues +1, +1, +1, -1. Its signature is (3,1).
 It is the Minkowski form, and Descartes' theorem is precisely the statement that the curvature 4-vector is
{\it null}. So there is a stronger reason for justifying the extension of Descartes' theorem to spinors than the simple observation of the correspondence between a linear-form and a quadratic form. In fact, a null vector in $\mathbb{R}(3, 1)$ is exactly the kind of object that has a
spinor square root, which is why Kocik's tangency spinor exists at all and why the
Apollonian group turns out to be generated by Lorentz reflections.}:
\[
Q_2(b_1, b_2, b_3, b_4):= \mathbf{x}^T \mathbf{Q_2 b} = (b_1^2+ b_2^2 + b_3^2+ b_4^2) - \frac{1}{2} (b_1+ b_2 + b_3+ b_4)^2.
\]
The subscript 2 in $\mathbf{Q}_2$ refers to the dimension of the space considered here. If $\mathbf{b} = (b_1\ b_2\ b_3\ b_4)^T$ denotes the column vector of curvatures and $\mathbf{c} = (b_1z_1, b_2z_2, b_3z_3, b_4z_4)^T$, then Descartes' theorem simply states that:
\[
\mathbf{b}^T\mathbf{Q_2b} = 0,
\]
and Descartes' "complex" theorem that:
\[
\mathbf{c}^T\mathbf{Q_2c} = 0.
\]
We can then easily add circles by acting on the matrix, changing the dimension of the quadratic form to express configurations in an $n$-dimensional Euclidean space, and we will then have:
\[
\mathbf{Q_n(b) := b}^T \mathbf{Q_nb} = 0,
\]
or move from circles to spheres, and even leave Euclidean space for spherical or hyperbolic geometries where we can also study Apollonian stacks, if desired.

Descartes' theorem is therefore generalized today in all possible ways\footnote{Let us note in passing that Clifford algebras would allow a much simpler and more elegant proof of Descartes' theorem than the one we gave above. As we have already said, this theorem indeed links the square of a linear form ($\sum k_i)^2$ to a quadratic form ($\sum k_j^2$). By interpreting the curvatures $k_i$ as the elements of a Clifford algebra, we can encode the relation directly, via anticommutators. Let us only sketch the proof here: let us consider a Clifford algebra $A$ with generators $e_1, e_2, e_3, e_4$ verifying $e_i^2 = 2$ and $e_ie_j +e_je_i = 0$ for $i \neq j$. Let us define $K = k_1e_1+k_2e_2 +k_3e_3+k_4e_4$. Let's calculate $K^2$:
\[
K^2 = (k_1^2 +k_2^2 + k_3^2 + k_4^2) \cdot 2 + 2(k_1k_2e_1e_2 + ...).
\]
If $K$ is isotropic (i.e $K^2 = 0$), then:
\[
\sum k_i^2 = 0 \quad \textnormal{et} \quad \sum k_ik_je_ie_j = 0.
\]
By imposing tangency (relations $k_4 = k_1+k_2+k_3 \pm 2 \sqrt{\sum k_ik_j}$), we find the theme.}.

\subsection{The Transition to Spinors}

As we suggested above, as soon as we have an equivalence between a quadratic form and the square of a linear form, we are in fact in the context that led, in the 19th century, to the emergence of Clifford algebras and, in the 20th century, with Dirac, to the expression of the relativistic Hamiltonian of the wave function, the Pauli matrices associated with particle spin, and finally, in their wake, the notion of "spinor," which is introduced quite naturally. We therefore necessarily find the same formalism here, hidden beneath Descartes' theorem.

This time, it was Kocik, initially working on Pythagorean triples (see \cite{Koc6} and also \cite{Par}), who introduced this particularly illuminating concept, that of the "tangency spinor."

Recall that spinors were introduced by Élie Cartan (1869-1951) in 1913 (see \cite{Car1}) and subsequently named by Paul Ehrenfest (1880-1933) for physical reasons. They have indeed been widely used by quantum mechanics in order to parameterize a new degree of freedom: this one appeared with the terms realizing the equivalence of quadratic forms and squares of linear forms in the Dirac equation, terms associated with matrices expressing a sort of rotation of the particles on themselves, the spin\footnote{Each spinor has $2n$ components, depending on whether $n = 2n + 1$ or $n = 2n$. When $n = 2$, the spinors of 4-dimensional space appear in the famous Dirac electron equations, the 4 wave functions being in fact the components of a spinor. The wave function of a fermion is therefore represented by a spinor. For particles with spin 1/2 (notably the electron), this is expressed by the Dirac equation. For hypothetical particles with spin 3/2, the Rarita-Schwinger equation would apply. Books abound on the concept of spinors, particularly by physicists and mathematician-physicists. These include \cite{Car2}, \cite{Str}, \cite{Pen}, and \cite{Hla}. But the origin of this concept is geometric.}.

Present in Cartan's work as part of a study of linear representations of simple groups, spinors also provide, from a mathematical point of view, a linear representation of the group of rotations of a space with any number of dimensions $n$. Let us briefly introduce them here in the case of two-component spinors:

Let two vectors $\mathbf{X}_1= (x_1, y_1, z_1)$ and $\mathbf{X}_2= (x_2, y_2, z_2)$ be in the Euclidean space $\mathbf{E}_3$, with the same origin $O$, mutually orthogonal, and of equal norm. These vectors define a plane, and if we consider them as ordered, this order defines a direction of rotation, in this case, that corresponding to the indices. Since the vectors are orthogonal and have the same magnitude, we can write:

\begin{align}
&||\mathbf{X_1}||^2 = ||\mathbf{X_2}||^2 = x_1^2+ y_1^2+ z_1^2 = x_2^2+y_2^2+z_2^2 \notag\\
& \mathbf{X_1.X_2} = x_1x_2+ y_1y_2+ z_1z_2 = 0.
\end{align}

To introduce an algebraic representation of the order, we multiply the components of the second vector by the imaginary number $i$, thus forming the three complex numbers:
\[
x = x_1+ix_2, \quad y = y_1+iy_2, \quad z = z_1+iz_2,
\]
thus constituting the components of a vector $\mathbf{Z} = (x, y, z)$ which, due to the choice of the previous vectors $\mathbf{X_1}$ and $\mathbf{X_2}$, are not independent, so we can write:
\[
\mathbf{Z.Z = ||Z||^2} = x^2+y^2+z^2 = 0,
\]
which makes $\mathbf{Z}$, a vector with zero norm, a so-called "isotropic" vector, orthogonal to itself.
We can then observe that the relationship between the three complex numbers $x, y$ and $z$ allows them to be expressed using only two complex numbers. Indeed, this relationship can be expressed in the form:
\[
z^2 = -(x^2+y^2) = -(x+iy)(x-iy),
\]
and if we set:
\[
x +iy = -2 \phi^2, \quad x -iy = 2 \psi^2,
\]
these two numbers $\phi$ and $\psi$ allow us to calculate $x, y$ and $z$, since we have:
\begin{equation}
x = \psi^2-\phi^2, \ y = i(\psi^2-\phi^2), \ z = \pm 2 \psi\phi \ \textnormal{(we generally choose the - sign)}
\end{equation}
The set of two complex numbers $\psi$ and $\phi$ constitutes a representation of the two vectors $\mathbf{X_1}$ and $\mathbf{X_2}$ as well as of the order chosen for them. The pair of numbers ($\psi,\phi)$, linked to the vectors $\mathbf{X_1}$ and $\mathbf{X_2}$ by relations (7), constitutes a {\it spinor}.

We introduced spinors via the 3-dimensional space, but they can exist in any dimension. Moreover, they are linked to the operation equating a quadratic form with a linear form because the coefficients that allow this equivalence, taken from a matrix algebra called the Clifford algebra, are generated by unitary matrices; any rotation in the reference space causes the spinors, via its associated matrix, to undergo a specific unitary transformation. Note that spinors are stable under addition and multiplication by a complex number, so they constitute intrinsic mathematical entities (expressing in particular the fact that a rotation of 2$\pi$ does not return to zero) and form a particular vector space. The same situations being linked to the same formalisms, in the case of Cartesian circles, any ordered pair of disks in the Euclidean plane, which are also capable of undergoing rotations, gives rise to a spinor $u \in \mathbb{R}^2$, an element of the spinor space of dimension 2 (symplectic space), defined up to sign. Algebraically, the process can be described as follows: the disks are represented by unit vectors of space type in the Minkowski space $M = \mathbb{R}^{3,1}$. A certain product of these vectors lies in the null cone of a subspace isomorphic to the Minkowski space $\mathbb{R}^{2,1} \subset M$. Such vectors can in turn be represented by the tensor square of a spinor, an element of the spinor space associated with $\mathbb{R}^{2,1}$.

It is, more specifically, the notion of a "tangency spinor" that will be useful here. A tangency spinor of an ordered pair of mutually tangent disks in the Euclidean plane (conveniently identified with the complex plane $\mathbb{R}2 \cong \mathbb{C}$) is a vector (concretely, a complex number) such that:
\[
u = \pm\sqrt{\frac{z}{r_1r_2}},
\]
where $z = \overrightarrow{O_1O_2}$ is the vector (expressed by the complex number) joining the centers $O_1$ and $O_2$ of the disks of radii $r_1$ and $r_2$, respectively. The tangency spinor is defined up to a sign.

This definition might seem artificial or arbitrary at first glance. However, as Kocik shows, it leads to a number of surprising and fruitful properties. For example, knowing that a quartet of mutually tangent disks is called a "Descartes configuration," Kocik demonstrates that the spinors in a Descartes configuration admit a choice of signs such that the following two properties hold:

\begin{equation}
\textnormal{"rot}\ \mathbf{u}=0\textnormal{"} : u_{12} + u_{23} + u_{31} = 0,
\end{equation}
\begin{equation}
\textnormal{"div}\ \mathbf{u} = 0\textnormal{"} : u_{14} + u_{24} + u_{34} = 0
\end{equation}
where $u_{i j}$ represents a spinor for the $i$-th and $j$-th tangent disks as shown in Fig. 6.

\begin{figure}[h] 
\vspace{-0.5\baselineskip}
\hspace{4\baselineskip}
\includegraphics[width=9in]{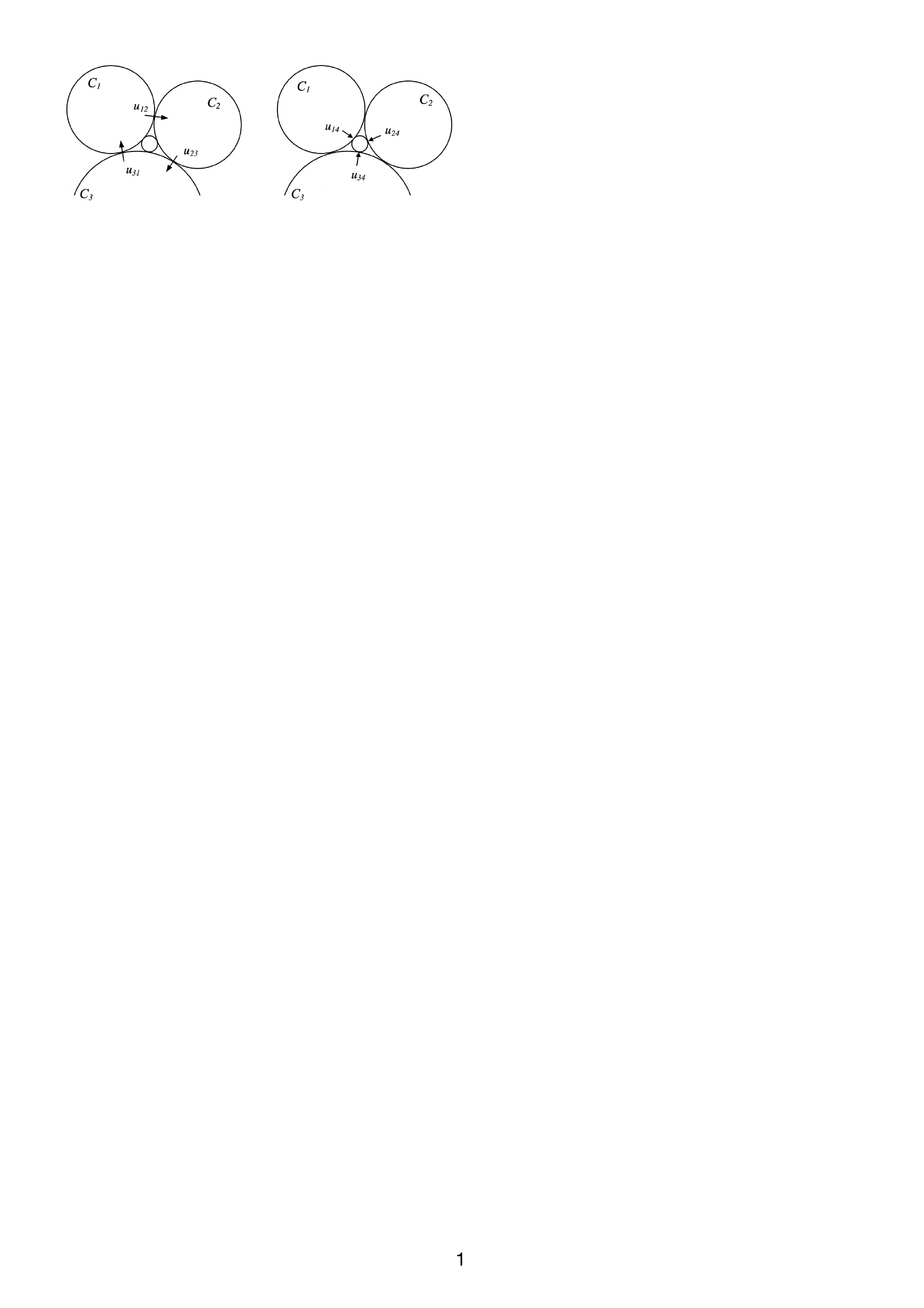}
\vspace{-53\baselineskip}
\caption{Spinor theorems according to \cite{Koc2}}
\label{fig: circle22}
\end{figure}

Kocik points out that the arrow representation of spinors in the figures is purely symbolic and only indicates the order of the disks. However, the last result mentioned above, "div $\mathbf{u} = 0$", can be considered a "spinorial" version of Descartes's theorem on circles, from which this one can also be deduced. Let's briefly demonstrate it here.

\begin{thm}
The relation \textnormal{(10)} implies Descartes' theorem.
\end{thm}
\begin{proof}
We start with the well-known relation (see, for example, \cite{Deh}):
\[
\mathbf{||{u•v}||^2 + ||u \times v||^2 = ||u||^2||v||^2},
\]
which is in fact, in disguised form, the Pythagorean theorem $cos^2 \theta + sin^2 \theta = 1$, known in three-dimensional vector calculus and valid in dimension 2. Consider the configuration and notation of Fig. 7.

\begin{figure}[h] 
\vspace{-1\baselineskip} 
\hspace{2\baselineskip} 
\includegraphics[width=9in]{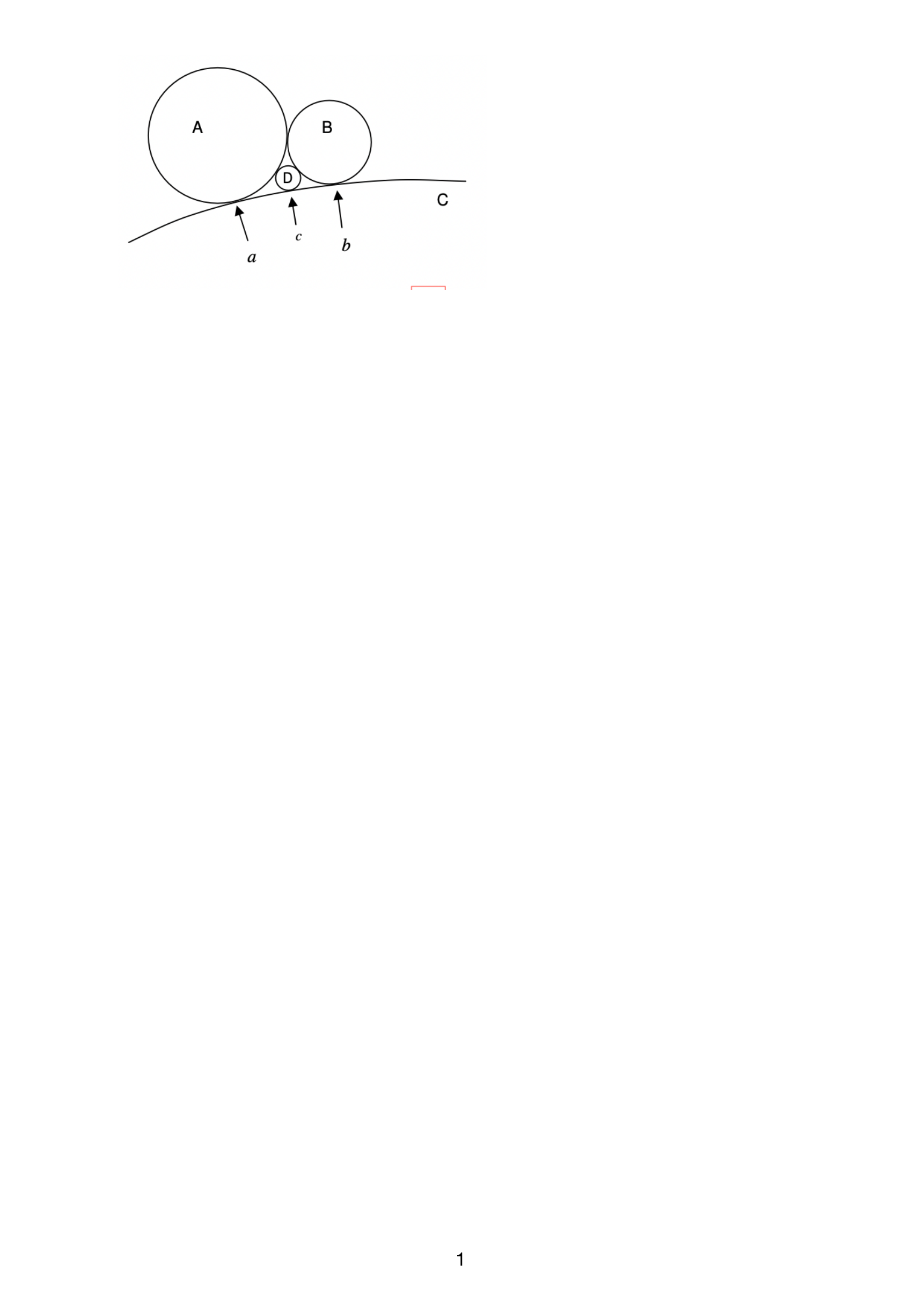} 
\vspace{-51\baselineskip} 
\caption{Proof of proposition 15 according to \cite{Koc2}} 
\label{fig: circle33} 
\end{figure}

We have:
\[
a+b=d \Rightarrow ||d||^2 =||a||^2 +2a \cdot b+|b|^2
\]
This gives the meaning of the scalar product:
\begin{align}
2a \cdot b & = ||d||^2 - ||a||^2 -||b||^2 \notag \\
& = C+D-(A+C)-B+C) \notag \\
&= D-A-B-C.
\end{align}
On the other hand, we can calculate the scalar product using:
\[
|a \cdot b|^2 = |a|^2|b|^2 - (a\times b)^2 = (A+C)(B+C)-C^2
= AB+BC+CA. \]
By equating the scalar product of two equations, we obtain:
\[
4(AB+BC+CA)^2 =(D-A-B-C)^2,
\]
which, after expansion and grouping of terms, is none other than Descartes' formula.
\end{proof}

In the case of Apollonian disk stacks, spinors are defined at each point of tangency of two disks, but we will not pursue this technical issue further and will stop our commentary on Kocik here.\footnote{The initial motivation for tangency spinors, from his perspective, was the discovery that the Apollonian Gasket (and other integral Apollonian disk stacks) contain Pythagorean triples. Spinors were thus introduced as the geometric representation of their Euclidean parameterizations. The name "spinor" is justified by the fact that these can be considered as spinors of Minkowski space $\mathbb{R}^{2+1}$, in a manner quite analogous to physical spacetime $\mathbb{R}^{3+1}$.}.

\subsection{Mathews's Approach}

To the knowledge of Mathews et al., none of the previously cited works provides a generalization of Descartes' circle theorem to what appear, in his vocabulary, as "$n$-flowers" for $n > 3$, although there are results involving flower-shaped configurations, such as the Soddy hexlet (see \cite{Sod1}). In fact, according to the authors, these works do not provide an explicit equation relating the curvatures.

Since "flowers" are simple examples of stacks of circles, general stacking theory applies to them (see \cite{Ste}). In general, as Mathews and al. explain, "from a simplicial 2-complex $K$ triangulating an oriented surface, the theory of circle stacking studies the existence and uniqueness of circle stacks realizing $K$, in the sense that the vertices of $K$ correspond to circles, the edges to tangencies and the triangles to triplets of oriented tangent circles. Flowers appear in the simple case where $K$ is a disk constructed from $n$ triangles around a central vertex" (see \cite{Mat}, 3). When $K$ is topologically a disk, Stephenson's "boundary value theorem" (see \cite{Ste}, thm. 11.6) applies, and the curvatures of the circles at the boundary vertices and the branching structure can be specified arbitrarily. There then exists a unique stacking (in Euclidean or hyperbolic geometry). The work of Mathews et al. then gives the Euclidean curvature of the inner circle in the case of the flower, as a root of an algebraic equation.

Starting from Descartes' classical circle theorem, which, in its modern interpretation, connects the curvatures of four mutually externally tangent circles, three "petal" circles around a central circle, forming a three-flower configuration, Mathews and Zymaris therefore attempt to generalize this theorem to the case of an "$n$-flower," consisting of $n$ tangent circles around a central circle, in order to give an explicit equation satisfied by their curvatures.

Without going into the details of such a construction (which assumes a spinor description of horospheres in hyperbolic geometry), let us nevertheless give a brief overview, if only to show how inspiring Descartes' theorem remains for contemporary mathematicians.

Let $n \ge 3$. An $n$-flower consists of a central circle $C_\infty$ and $n$ petal circles $C_j$, over integers $j$ mod $n$, such that the $j$s are externally tangent to $C_\infty$ in order around $C_\infty$, and each $C_j$ is externally tangent to $C_{j-1}$ and $C_{j+1}$.

In this article, the authors denote the curvature of a circle $C_\bullet$ by $\kappa_\bullet$. Descartes' circle theorem gives an equation satisfied by the curvatures of a three-dimensional flower, and in this case we have:
\begin{equation}
(\kappa_\infty + \kappa_1 + \kappa_2 + \kappa_3)^2 = 2(\kappa_\infty^2 + \kappa_1^2 + \kappa_2^3 + \kappa_3^2).
\end{equation}
This therefore requires presenting an explicit equation satisfied by the curvatures of an $n$-dimensional flower. This generalization of Descartes' theorem then takes the following form:

Let $C_\infty$ and $C_j (j \in \mathbb{Z}/n\mathbb{Z})$ be the circles of an $n$-flower having curvatures $\kappa_\infty, \kappa_j$ respectively. The authors define $m_0$ and $m_j$ for $1 \leq j \leq n - 1$ as:
\begin{equation}
m_0 = \sqrt\frac{K_0}{k_\infty}+1, m_j = \sqrt(\frac{K_j}{k_\infty}+1)(\frac{K_{j-1}}{k_\infty}+1))-1. \end{equation}

So, for $n$ odd, we have:
\begin{equation}
\frac{m_0^2}{2}(\prod_{j=1}^{n-1} (m_j-i) - \prod_{j=1}^{n-1} (m_j+i))-\prod_{j=1}^{\frac{n-1}{2}}(m^2_{2j}+1 =0.
\end{equation}
And for $n$ even, the following equation holds:
\begin{equation}
\frac{i}{2}(\prod_{j=1}^{n-1} (m_j-i) - \prod_{j=1}^{n-1} (m_j+i))-\prod_{j=1}^{\frac{n-1}{2}}(m^2_{2j-1}+1 =0.
\end{equation}

In other words, from the curvatures $\kappa_\bullet$, the authors define auxiliary variables $m_\bullet$ via (13), and the $m_\bullet$ satisfy polynomial equations. Since every $\kappa_\bullet > 0$, each $m_\bullet$ is the square root of a clearly positive number, and they take $m_\bullet$ as the positive square root.
In (14) and (15), $i$ is the usual square root of $-1$. Although complex numbers are essential to the proof, for writing these equations, they are mostly a convenience. After expansion and cancellation of terms, the resulting polynomials have integer coefficients. Indeed, by writing [$n$] for $\{1,2,...,n\}$, then for any subset $K \subseteq [n - 1]$, the term of the product $\prod^{n-1}_{j=1}(m_j\pm i)$ including precisely the $m_k$ with $k \in K$ is given by $(\pm i)^{n-1-|K}, \prod_{k \in K} m_k$. These terms come in pairs, one from $\prod_{j=1}^{n-1} (m_j-i)$ and the other from $\prod_{j=1}^{n-1} (m_j+i)$. When $n-1-|K|$ is even, the terms of each pair are real, equal and cancel each other out; when $n-1-|K|$ is
odd, that is, equal to $2l+1$, then the terms of each pair are imaginary and conjugate. Thus, since $\frac{i}{2}((-i)^{n-1-|K|}-i^{n-1-|K|}) = (-1)^l$ (14) and (15) can be written in the form;

\begin{equation}
m_0^2 = \sum\limits_{\underset{K \subseteq [n-1] |K| = n-2l-2}{K \subseteq [n-1]}} (-1)^l \prod_{k \in K} m_k = \prod_{j =1}^{\frac{n-1}{2}} (m_{2-1j}^2 +1)
\end{equation}
And:
\begin{equation}
 \sum\limits_{\underset{K \subseteq [n-1] |K| = n-2l-2}{K \subseteq [n-1]}} (-1)^l \prod_{k \in K} m_k = \prod_{j =1}^{\frac{n-1}{2}} (m_{2j}^2 +1)
\end{equation}
respectively.

In any case, making the substitutions of (13) into (14) or (15) (or (16) or (17)) gives an equation satisfied by the $\kappa_\bullet$,
providing a generalization of Descartes' equation (15). This equation involves square roots, but by multiplying by various conjugates (replacing various $m_\bullet$ with $-m_\bullet$), we can obtain a polynomial relationship between the $m^2_\bullet$; after substitution and erasure of the denominators, we obtain a polynomial relationship between the $\kappa_\bullet$.
In this way, from equation (15), with $n = 3$, we can find the Cartesian circle theorem. Similarly, from (15), with $n = 4$, we obtain the following equation for the curvatures in a 4-flower:

\begin{align}
16\kappa_\infty^4 - 8\kappa_\infty^2 (\kappa_1\kappa_2 &+ \kappa_2\kappa_3 + \kappa_3\kappa_4 + \kappa_4\kappa_1 + 2 \kappa_1\kappa_3 + 2\kappa_2\kappa_4) + (\kappa_1^2 + \kappa_3^2)(\kappa_2^2+ \kappa_4^2) \notag\\
& - 16 \kappa_\infty(\kappa_1\kappa_2 \kappa_3 + \kappa_2\kappa_3\kappa_4 + \kappa_3\kappa_4\kappa_1 + \kappa_4\kappa_1\kappa_2) - 12\kappa_1\kappa_2\kappa_3\kappa_4 \notag \\
& - 2(\kappa_1\kappa_2+ \kappa_3\kappa_4)(\kappa_2\kappa_3+ \kappa_4\kappa_1) = 0.
\end{align}

The larger $n$ is, obviously, the faster the polynomials grow in size and degree.

"Flowers" are thus a building block of circular packing theory. The general theory of circular packing shows that once the curvature of the petals of a flower of type $n$ is known, the curvature of the central circle is determined. Descartes' generalized theorem provides an equation for determining this central curvature.

Based on the work of Mathews (see \cite{Mat1}, Penrose-Rindler (see \cite{Pen}) and Penner (see \cite{Penn}), the demonstration of this result is not obvious and requires some rather subtle reasoning. But its outcome is an incredible generalization of Descartes' theorem, in other words an extension of the simple consideration of 3-flowers to that of $n$-flowers ($n$ tangent circles around the central circle) with an explicit equation satisfied by the curvatures. A final bouquet.

\section{Conclusion}

It may seem futile to the philosopher to be interested in a theorem of Descartes that does not seem to present much interest, {\it a priori}, for his philosophy, nor for philosophy in general. No doubt the philosopher exiled in Egmond aan den Hoef  (a village near the sea, in the Dutch province of North Holland) was initially interested in it as a mathematician and submitted it to Princess Elizabeth as an exercise. But does one do this kind of thing absolutely gratuitously when one is a philosopher? And does one reflect on contact without ulterior motive when one admires (if not a little in love with) a person of a completely different background than one's own? How can one find the right distance to be and remain in contact without transgressing the social boundaries of decorum? Throughout this correspondence, the obsequious Descartes lost himself in idle circumlocutions ("the very humble and very obedient servant"...) while his correspondent allowed herself to sign "your very affectionate friend at your service" and, speaking of a man who had spent most of his life pursuing designs of love and ambition, admitted to never having "stronger and more constant" ones with the philosopher than that of this affectionate friendship. Mathematics is generally more direct, and in the eyes of a psychoanalyst, "kissing circles" could well pass, not as always overly hyperbolic metaphors, but, to borrow Viète's phrase, a way of "finely touching" Elizabeth. 

But let us put aside these speculations, which are unconfirmed. More seriously, the philosophical question at the center of the relationship between Descartes and Elizabeth, as we know, is that of the union of soul and body. How can there be a union of soul and body if both are substances between which there can only be a difference, by definition, "substantial," that is, real and absolute? But of course, there is a third, the soul-body compound, and, moreover, a point of contact between soul and body, which is, moreover, a little more than a point, a sort of small roundness: the pineal gland. The soul, the body, and the union of the two are therefore like circles manifesting points of tangency between them and also with this famous central pineal gland. Perhaps it was to bring this knowledge of contacts into focus that the philosopher submitted this Apollonian problem to the Princess Palatine? It seems to me less false to consider this than to assume, as has been done, that the tree of philosophy is not a tree, but rather bends back on itself to the point of transplanting the branch of morality into the roots, a bushy or rhizomatic solution that no longer has much to do with Descartes (see \cite{Vaq}). On the other hand, any tree in the mathematical sense can also be considered isomorphic to a series of circles (or ellipses) (see \cite{Par2}): a central circle for the trunk (physics and its metaphysical roots), and osculating circles to represent the different branches (in this case, medicine, mechanics, and morality). This is also a possible interpretation of this Cartesian reflection on the circle theorem.

But beyond Descartes and the history of philosophy, as has been said, the circle theorem has an undeniable epistemological interest. Through his successive generalizations, he testifies to the inescapable fact that certain scientific truths are not only inscribed forever in history, but that they are capable of multiple extensions, and that we must therefore grant them an inductive value: not only are they what they are (so to speak "{\it auto kath hauto}", like Platonic ideas), but they are capable of producing offspring, of generating even vaster and more encompassing truths, which do not, however, contravene the core that gave rise to them, but, on the contrary, confirm it.

\end{document}